\documentclass[oneside,final,11pt]{amsart}

\usepackage{dsfont}

\newcommand{\dd}{\mathrm d}

\newcommand{\Id}{\mathrm{Id}}
\newcommand{\KS}{K_{\mathcal S}}
\newcommand{\tauS}{\tau_{\mathcal S}}
\newcommand{\omegaS}{\omega_1^{\mathcal S}}
\newcommand{\bigcupdot}[1]{\;\cdot\hspace{-8pt}\bigcup_{#1}}
\newcommand{\upstar}[1]{{#1}^{\raise1pt\hbox{$\scriptscriptstyle*$}}}
\newcommand{\chilow}[1]{\chi_{\lower2pt\hbox{$\scriptstyle#1$}}}

\DeclareMathOperator{\Ker}{Ker}

\title[Isometric copies of continuous functions spaces]{Extension property and complementation of isometric copies of continuous functions spaces}

\author{Claudia Correa}
\thanks{The first author is sponsored by FAPESP (Process no.\ 2012/25171-0).}
\address{Departamento de Matem\'atica,\hfill\break\indent Universidade de S\~ao Paulo, Brazil}
\email{claudiac.mat@gmail.com}
\author{Daniel V. Tausk}
\address{Departamento de Matem\'atica,\hfill\break\indent Universidade de S\~ao Paulo, Brazil}
\email{tausk@ime.usp.br} \urladdr{http://www.ime.usp.br/\~{}tausk}

\date{October 15th, 2013}

\begin{document}

%\swapnumbers
\theoremstyle{plain}\newtheorem{teo}{Theorem}[section]
\theoremstyle{plain}\newtheorem{prop}[teo]{Proposition}
\theoremstyle{plain}\newtheorem{lem}[teo]{Lemma}
\theoremstyle{plain}\newtheorem{cor}[teo]{Corollary}
\theoremstyle{definition}\newtheorem{defin}[teo]{Definition}
\theoremstyle{remark}\newtheorem{rem}[teo]{Remark}
\theoremstyle{plain} \newtheorem{assum}[teo]{Assumption}
\theoremstyle{definition}\newtheorem{example}[teo]{Example}

\begin{abstract}
In this article we prove that every isometric copy of $C(L)$ in $C(K)$ is complemented if $L$ is compact Hausdorff of finite height and $K$
is a compact Hausdorff space satisfying the {\em extension property}, i.e., every closed subset of $K$ admits an extension operator. The space
$C(L)$ can be replaced by its subspace $C(L|F)$ consisting of functions that vanish on a closed subset $F$ of $L$. We also study the class
of spaces having the extension property, establishing some closure results for this class and relating it to other classes of compact
spaces.
\end{abstract}

\maketitle

\begin{section}{Introduction}

In this article we prove a result (Theorem~\ref{thm:main}) concerning the complementation of isometric copies of $C(L)$ in $C(K)$, for compact Hausdorff spaces
$K$ and $L$ satisfying certain assumptions. Our result concerns, in particular, the complementation of Banach subalgebras of $C(K)$ (Corollary~\ref{thm:complsubalgebra}).
Indeed, recall that a Banach subalgebra (with unity) of $C(K)$ is the image $\phi^*C(L)$ of the isometric embedding
$\phi^*:C(L)\ni f\mapsto f\circ\phi\in C(K)$, where $\phi:K\to L$ is a continuous surjection. The Banach
subalgebra $\phi^*C(L)$ is complemented in $C(K)$ if and only if $\phi$ admits an {\em averaging operator}, i.e., a bounded left
inverse $P:C(K)\to C(L)$ for the composition operator $\phi^*$. Finding conditions under which a continuous surjection admits an averaging operator
is a classical problem in Banach space theory (\cite{AA, Ditor, PelDisser}).

More precisely, Theorem~\ref{thm:main} deals with the complementation of isometric copies in $C(K)$ of the subspace $C(L|F)$ of $C(L)$
consisting of continuous functions vanishing on a closed subset $F$ of $L$.

We observe that the special case of Corollary~\ref{thm:complsubalgebra} when $K$ and $L$ are compact lines and $\phi:K\to L$ is a continuous {\em increasing\/} surjection is a corollary of Lemma~2.7 of \cite{KK}, though the general case (when $\phi$ is not necessarily increasing) is new. By a {\em compact line\/} we mean a linearly ordered set which is compact in the order topology.

The assumptions about the spaces $K$ and $L$ in Theorem~\ref{thm:main} are the following: the space $L$ is assumed to be scattered of finite
(Cantor--Bendixson) height and the space $K$ is required to satisfy a condition which we call the extension property. We say that $K$ has the {\em extension property\/} if every closed subset $F$ of $K$ admits an {\em extension operator\/} in $K$,
i.e., a bounded operator $E_F:C(F)\to C(K)$ such that $E_F(f)$ is an extension of $f$, for all $f\in C(F)$.

\smallskip

We note that if $K$ and $L$ satisfy the assumptions of Theorem~\ref{thm:main} and $C(K)$ contains an isometric copy of $C(L)$ then the space $C(L)$ is actually isomorphic
to $c_0(I)$, for some index set $I$ (Corollary~\ref{thm:corHol}). Thus, our Theorem~\ref{thm:main} can be seen as a Sobczyk-like theorem. 
The complementation of isomorphic copies of $c_0(I)$ in Banach spaces has been studied in \cite{ArgyrosLondon}. More precisely, in \cite[Theorem~1.1]{ArgyrosLondon}
it is shown that every isomorphic copy of $c_0(I)$ in a Banach space $X$ is complemented if $\vert I\vert<\aleph_\omega$ and $X$ is isomorphic to a subspace of a $C(K)$ space, with $K$ a Valdivia compactum. We observe that our Theorem~\ref{thm:main} does not follow from \cite[Theorem~1.1]{ArgyrosLondon},
even when $C(L)\cong c_0(I)$ with $\vert I\vert<\aleph_\omega$.
Namely, if $K$ is a separable nonmetrizable compact line (say, the double arrow space) then $K$ has the extension property, yet $X=C(K)$ does not satisfy the assumptions of \cite[Theorem~1.1]{ArgyrosLondon}. The latter statement
follows from \cite[Theorem~2.1]{CT} and from the fact that the space of continuous functions on a Valdivia compactum has the separable complementation
property (\cite[Lemma, pg.~494]{Valdivia2}).

\medskip

This article is organized as follows. Section~\ref{sec:CLCK} is devoted to the proof of Theorem~\ref{thm:main}. Section~\ref{sec:extensor} is
devoted to the study of the extension property. We start Section~\ref{sec:extensor} by stating results that are immediate consequences of well-known facts.
Namely, every metrizable compact space and every compact line have the (regular) extension property (Propositions~\ref{thm:propmetrizable} and \ref{thm:propline}).
Moreover, the extension property is local (Corollary~\ref{thm:corlocal}). Then we study some closure properties of the class of spaces having the
extension property. We leave open the question of whether this class is closed under continuous images, but we prove some results
in this direction (Proposition~\ref{thm:regavgop} and Corollary~\ref{thm:continuousimage}).
We prove also that the class of spaces having the regular extension property is closed under the operation of taking Alexandroff compactifications
of arbitrary topological sums (Proposition~\ref{thm:assembler}). It follows that all spaces that can be assembled from
compact lines and compact metrizable spaces by iterating this operation have the extension property.
We conclude the section with examples of specific spaces and of classes of spaces that do not have the extension property.
More specifically, we show that if $K$ is a ccc space with uncountable spread or if ($K$ is infinite and) $C(K)$ has the Grothendieck property
then $K$ does not have the extension property.

\end{section}

\begin{section}{Complementation of isometric copies of $C(L)$ in $C(K)$}
\label{sec:CLCK}

As usual, $C(K)$ denotes the Banach algebra of real-valued continuous functions on the compact Hausdorff space $K$, endowed with
the supremum norm, and $\mathbf1_K$ denotes its unity. If $F$ is a closed subset of $K$, we denote by $\rho_F:C(K)\to C(F)$ the restriction operator
$f\mapsto f\vert_F$ and by $C(K|F)$ its kernel.
\begin{defin}
Given a compact Hausdorff space $K$ and a closed subset $F$ of $K$, by an {\em extension operator\/} for $F$ in $K$ we mean a bounded linear map
$E_F:C(F)\to C(K)$ that is a right inverse for $\rho_F$. We say that $E_F$ is {\em regular\/} if $\Vert E_F\Vert\le1$ and $E_F(\mathbf1_F)=\mathbf1_K$
(equivalently, $E_F$ is regular if $E_F(\mathbf1_F)=\mathbf1_K$ and $E_F$ is a positive operator).
We say that $K$ has the {\em extension property\/} (resp., {\em regular extension property}) if every nonempty closed subset of $K$ admits an extension
operator (resp., a regular extension operator) in $K$.
\end{defin}
Recall that a bounded operator admits a bounded right inverse if and only if it is onto and its kernel is complemented. Therefore, a closed subset $F$ of
$K$ admits an extension operator in $K$ if and only if $C(K|F)$ is complemented in $C(K)$.
Obviously, the (regular) extension property is hereditary to closed subspaces.

\medskip

For the proof of the main result of this section, we need two lemmas.
\begin{lem}\label{thm:complementacao}
Let $X$ and $Y$ be Banach spaces, $T:X\to Y$ be a bounded operator, and $Z$ be a closed subspace of $X$. If\/ $T[Z]$ is closed and complemented in $Y$
and $Z\cap\Ker(T)$ is complemented in $X$ then $Z$ is complemented in $X$.
\end{lem}
\begin{proof}
Let $P$ be a projection of $X$ onto $Z\cap\Ker(T)$ and $Q$ be a projection of $Y$ onto $T[Z]$. The operator $(\Id_X-P)\vert_Z:Z\to Z$ passes
to the quotient through $T\vert_Z:Z\to T[Z]$, i.e., there exists a bounded operator $M:T[Z]\to Z$ with $M\circ T\vert_Z=(\Id_X-P)\vert_Z$.
It is easily checked that $M\circ Q\circ T+P$ is a projection of $X$ onto $Z$.
\end{proof}

\begin{lem}\label{thm:c0I}
Let $K$ be a compact Hausdorff space and $S:c_0(I)\to C(K)$ be an isometric immersion, where $I$ is a set. Let $\big\{\xi_i:i\in I\big\}\subset C(K)$ be
the image under $S$ of the canonical basis of $c_0(I)$ and set $G=\bigcap_{i\in I}\xi_i^{-1}(0)$. Then the image of $S$ is complemented in $C(K|G)$.
\end{lem}
\begin{proof}
Since $S$ is an isometric immersion, we have $\Vert\xi_i\Vert=1$ and thus, for each $i\in I$, there exists $x_i\in K$ with
$\vert\xi_i(x_i)\vert=1$. The fact that $S$ is an isometric immersion yields also that $\Vert\xi_i\pm\xi_j\Vert=1$ for $i\ne j$ and therefore
$\xi_j(x_i)=0$ for $i\ne j$. It follows that the points $x_i$ are distinct and that all the limit points of the set $\big\{x_i:i\in I\big\}$
are in $G$. This implies that, for $\xi\in C(K|G)$, the family $\big(\xi(x_i)\big)_{i\in I}$ is in $c_0(I)$. Hence the map $P:C(K|G)\to c_0(I)$ defined by
$P(\xi)=\big(\xi_i(x_i)\xi(x_i)\big)_{i\in I}$ is a bounded left inverse for $S$ and thus the image of $S$ is complemented in $C(K|G)$.
\end{proof}

\begin{teo}\label{thm:main}
Let $K$ be a compact Hausdorff space having the extension property. If $L$ is a compact Hausdorff scattered space of finite height and $F$ is a closed
subset of $L$ then every isometric copy of $C(L|F)$ in $C(K)$ is complemented.
\end{teo}
\begin{proof}
First, we observe that we can restrict ourselves to the case when $F$ has at most one point. Namely, if $F$ has more than one point then
$C(L|F)$ is isometric to $C(L_1|F_1)$, where $L_1$ is the quotient of $L$ obtained by collapsing $F$ to a single point $p\in L_1$ and $F_1=\{p\}$.
Note that $L_1$ is again scattered of finite height.

In what follows, either $F=\emptyset$ or $F=\{p\}$.
We will use induction on the height of $L$. Denote by $L'$ the Cantor--Bendixson derivative of $L$.
The case when $F=\{p\}$ with $p\in L\setminus L'$ can be reduced to the case when $F$ is empty
since $C(L|\{p\})$ is isometric to $C(L\setminus\{p\})$.

Let $S:C(L|F)\to C(K)$ be an isometric immersion and denote the restriction of $\rho_{L'}:C(L)\to C(L')$ to $C(L|F)$
by $R:C(L|F)\to C(L'|F)$. Setting $I=L\setminus L'$, then $f\mapsto f\vert_I$ is an isometry from $C(L\vert L')$ onto $c_0(I)$.
For $i\in I$, denote by $f_i\in C(L|L')$ the characteristic function of $\{i\}$ and set $\xi_i=S(f_i)$.

Define $G$ as in the statement of Lemma~\ref{thm:c0I}. The map $\rho_G\circ S$ annihilates the set $\big\{f_i:i\in I\big\}$ and thus
also its closed span $C(L|L')=\Ker(R)$. It follows that $\rho_G\circ S$ passes to the quotient through $R$, i.e., there exists a bounded operator $\overline S:C(L'|F)\to C(G)$
such that $\overline S\circ R=\rho_G\circ S$.

We will prove
in a moment that $\overline S$ is an isometric immersion. Assuming that this is the case, we will conclude the proof by applying
Lemma~\ref{thm:complementacao} with $T=\rho_G$ and with $Z$ equals to the image of $S$. Note that $T[Z]$ equals the image of $\overline S$, which is (closed and)
complemented in $C(G)$, by the induction hypothesis. Moreover, $Z\cap\Ker(T)=S[C(L|L')]$ is complemented in $C(K|G)$ by Lemma~\ref{thm:c0I} and
$C(K|G)$ is complemented in $C(K)$, because $K$ has the extension property.

Now let us prove that $\overline S$ is an isometric immersion. Let $g\in C(L'|F)$ with $\Vert g\Vert=1$ be fixed. Then $\overline S(g)=S(f)\vert_G$,
where $f\in C(L)$ is an arbitrarily chosen extension of $g$. To prove that $\Vert\overline S(g)\Vert=1$,
we will show that $f$ can be chosen such that $\Vert f\Vert=1$ and such that $S(f)$ attains its maximum modulus on $G$.
Consider first the case when $g$ admits an extension $f\in C(L)$ with $\vert f\vert<1$ on $I$. Assume by contradiction that $S(f)$ attains its maximum
modulus at a point $x\in K\setminus G$. Then $\xi_i(x)\ne0$, for some $i\in I$. For $\varepsilon>0$ small enough, we have $\Vert f\pm\varepsilon f_i\Vert=1$
and
\[1=\Vert S(f\pm\varepsilon f_i)\Vert\ge\vert S(f)(x)\pm\varepsilon\xi_i(x)\vert>1.\]
Finally, let $f$ be an extension of $g$ with $\Vert f\Vert=1$. We can assume that the set
\[J=\big\{i\in I:\vert f(i)\vert=1\big\}\]
is infinite,
otherwise $f$ could be modified so that $\vert f\vert<1$ on $I$. For $i\in J$, we have $\Vert f\pm f_i\Vert=2$ and thus there exists
$x_i\in K$ with $\vert S(f)(x_i)\pm\xi_i(x_i)\vert=2$. Since $\vert S(f)(x_i)\vert\le1$ and $\vert\xi_i(x_i)\vert\le1$, we must have
$\vert S(f)(x_i)\vert=1$ and $\vert\xi_i(x_i)\vert=1$. As in the proof of Lemma~\ref{thm:c0I}, the points $x_i$, $i\in J$, are distinct and the
limit points of $\big\{x_i:i\in J\big\}$ are in $G$. Taking a limit point $x$ of the infinite set $\big\{x_i:i\in J\big\}$, then
$x\in G$ and $\vert S(f)(x)\vert=1$, concluding the proof.
\end{proof}

\begin{cor}\label{thm:complsubalgebra}
Let $K$ and $L$ be compact Hausdorff spaces and $\phi:K\to L$ be a continuous surjective map. If $K$ has the extension property and
$L$ is scattered of finite height then the Banach subalgebra $\phi^*C(L)$ of $C(K)$ is complemented (equivalently, $\phi$ admits
an averaging operator).\qed
\end{cor}

The following result is an immediate consequence of Theorem~\ref{thm:main} obtained by letting $L=I\cup\{\infty\}$ be the one-point compactification of the
discrete space $I$ and by setting $F=\{\infty\}$. Such result has been stated in \cite[Theorem~2.1]{Patterson} in the case when $I$ is countable and $K$ is the double arrow space, though the proof that appears in \cite{Patterson} works also in the general case.
\begin{cor}\label{thm:corc0I}
If $K$ is a compact Hausdorff space having the extension property then every isometric copy of $c_0(I)$ in $C(K)$ is complemented.\qed
\end{cor}

\begin{rem}\label{thm:remcontinuousimage}
The class of compact Hausdorff spaces $K$ for which the thesis of Theorem~\ref{thm:main} holds is clearly closed under continuous images.
\end{rem}

\end{section}

\begin{section}{The extension property}\label{sec:extensor}

In this section we study the class of compact Hausdorff spaces having the (regular) extension property. We start by collecting some well-known
results about the existence of extension operators.

\medskip

In \cite[\S6]{PelDisser}, a compact Hausdorff space $K$ is called a {\em Dugundji space\/} if for every compact Hausdorff space $L$
containing (a homeomorphic copy of) $K$, there exists a regular extension operator for $K$ in $L$. It is well-known that every nonempty
compact metrizable space is a Dugundji space (\cite[Theorem~6.6]{PelDisser}). The following proposition is an immediate consequence of this fact.
\begin{prop}\label{thm:propmetrizable}
Every compact metrizable space has the regular extension property.\qed
\end{prop}

In a compact line, a regular extension operator for a nonempty closed subset can be constructed by decomposing its complement into disjoint
open intervals and by using Urysohn's Lemma. One thus obtain the following result.
\begin{prop}\label{thm:propline}
Every compact line has the regular extension property.
\end{prop}
\begin{proof}
See \cite[Lemma~4.2]{Kubis}.
\end{proof}

A standard argument using partitions of unity shows that the (regular) extension property is local.
\begin{prop}
Let $K$ be a compact Hausdorff space and $F$ be a nonempty closed subset of $K$. If every point $x\in F$ has a closed neighborhood
$V$ in $K$ such that $V\cap F$ admits an extension operator (resp., a regular extension operator) in $V$ then $F$ admits an extension operator
(resp., a regular extension operator) in $K$.
\end{prop}
\begin{proof}
See \cite[Lemma~3.6]{PelDisser}.
\end{proof}

\begin{cor}\label{thm:corlocal}
Let $K$ be a compact Hausdorff space. If every point of $K$ has a closed neighborhood having the (resp., regular) extension property then
$K$ has the (resp., regular) extension property.\qed
\end{cor}

\subsection{Closure properties of the class of spaces having the extension property}
We will first establish that, under some additional assumptions, a continuous image of a space having the extension property also has the extension
property (Proposition~\ref{thm:regavgop} and Corollary~\ref{thm:continuousimage}). Then, we show that the class of spaces having the regular extension
property is closed under the operation of taking Alexandroff compactifications of arbitrary topological sums (Proposition~\ref{thm:assembler}).

\begin{prop}\label{thm:regavgop}
Let $\phi:K\to L$ be a continuous surjection, where $K$ and $L$ are compact Hausdorff spaces. Assume that $\phi$ admits a regular averaging operator
$P:C(K)\to C(L)$ (i.e., $P$ is an averaging operator with $\Vert P\Vert\le1$). If $K$ has the (resp., regular) extension property then
also $L$ has the (resp., regular) extension property.
\end{prop}
\begin{proof}
Given a nonempty closed subset $F$ of $L$, let $G=\phi^{-1}[F]$ and $E_G$ be a (regular) extension operator for $G$ in $K$. Set
\[E_F=P\circ E_G\circ(\phi\vert_G)^*:C(F)\longrightarrow C(L).\]
That $E_F$ is indeed an extension operator follows easily using the fact that, for all $x\in L$, there exists a probability measure
$\mu_x$ on $\phi^{-1}(x)$ such that $P(h)(x)=\int_{\phi^{-1}(x)}h\,\dd\mu_x$, for all $h\in C(K)$ (see \cite[Proposition~4.1]{PelDisser}).
\end{proof}

The following definition and lemma are required for the proof of Proposition~\ref{thm:willfollow}, from which Corollary~\ref{thm:continuousimage}
will follow.
\begin{defin}
Given a compact Hausdorff space $K$, a closed subset $F$ of $K$, and a closed subset $G$ of $F$, by an {\em extension operator
modulo $G$\/} for $F$ in $K$ we mean a bounded operator which is a right inverse for the map
$\rho_F\vert_{C(K|G)}:C(K|G)\to C(F|G)$.
\end{defin}

\begin{lem}\label{thm:lemamoduloG}
Let $K$ be a compact Hausdorff space, $F$ be a closed subset of $K$, and $G$ be a closed subset of $F$.
If $F$ admits an extension operator modulo $G$ in $K$ and $G$ admits an extension operator in $K$ then $F$ admits an extension operator in $K$.
\end{lem}
\begin{proof}
Since $F$ admits an extension operator modulo $G$ in $K$, the subspace $C(K|F)$ is complemented in $C(K|G)$. Moreover, since $G$ admits an extension
operator in $K$, the subspace $C(K|G)$ is complemented in $C(K)$. Hence $C(K|F)$ is complemented in $C(K)$.
\end{proof}

\begin{prop}\label{thm:willfollow}
Let $K$ be a compact Hausdorff space and denote by $K'$ its Cantor--Bendixson derivative. If $K'$ has the extension property and $K'$ admits an extension
operator in $K$ then $K$ has the extension property.
\end{prop}
\begin{proof}
If $F$ is a closed subset of $K$ then $F\cap K'$ admits an extension operator in $K$. Moreover, the map $E:C(F|F\cap K')\to C(K|F\cap K')$
defined by $E(f)\vert_F=f$, $E(f)\vert_{K\setminus F}\equiv0$ is an extension operator modulo $F\cap K'$ for $F$ in $K$. The conclusion
follows from Lemma~\ref{thm:lemamoduloG}.
\end{proof}

\begin{cor}\label{thm:continuousimage}
Let $K$ be a compact Hausdorff scattered space of finite height. If $K$ is a continuous image of a compact Hausdorff space having the extension property
then $K$ has the extension property and the space $C(K)$ is isomorphic to $c_0(I)$, for some index set $I$.
\end{cor}
\begin{proof}
Note that $C(K|K')$ is isometric to $c_0(K\setminus K')$ and therefore it is complemented in $C(K)$, by Remark~\ref{thm:remcontinuousimage}.
Thus, $K'$ admits an extension operator in $K$ and $C(K)\cong c_0(K\setminus K')\oplus C(K')$. The conclusion is obtained using induction on the height of $K$.
\end{proof}

\begin{cor}\label{thm:corHol}
If $K$ is a compact Hausdorff space having the extension property and $L$ is a compact Hausdorff scattered space of finite height such that $C(K)$ contains
an isometric copy of $C(L)$ then $L$ has the extension property and the space $C(L)$ is isomorphic to $c_0(I)$, for some index set $I$.
\end{cor}
\begin{proof}
Simply observe that, by Holszty\'nski's Theorem \cite{Holsztynski}, if $C(K)$ contains an isometric copy of $C(L)$ then $L$ is a continuous image of a closed
subspace of $K$.
\end{proof}

As an immediate consequence of Proposition~\ref{thm:willfollow}, we obtain a useful criterion for verifying if a space of finite height has the
extension property.
\begin{cor}
Let $K$ be a compact Hausdorff scattered space of finite height and denote by $K^{(n)}$ its $n$-th Cantor--Bendixson derivative.
If for all $n$, $K^{(n+1)}$ admits an extension operator in $K^{(n)}$ then $K$ has the extension property.\qed
\end{cor}

By the {\em Alexandroff compactification\/} of a locally compact Hausdorff space $X$ we mean $X$ itself, if $X$ is compact, or its one-point
compactification $X\cup\{\infty\}$, if $X$ is not compact.
\begin{prop}\label{thm:assembler}
The class of compact Hausdorff spaces having the regular extension property is closed under the operation of taking the Alexandroff compactification
of arbitrary topological sums.
\end{prop}
\begin{proof}
Let $(K_i)_{i\in I}$ be a family of nonempty compact Hausdorff spaces having the regular extension property. Denote by $\bigcupdot{i\in I}K_i$ its
topological sum and by $K$ the Alexandroff compactification of the latter. Let $F$ be a nonempty closed subset of $K$ and set $F_i=F\cap K_i$.
When $F_i\ne\emptyset$, let $E_i$ be a regular extension operator for $F_i$ in $K_i$. Let us define a regular extension operator $E_F$
for $F$ in $K$.

If $F$ is contained in $\bigcupdot{i\in I}K_i$, choose an arbitrary $p\in F$
and define $E_F$ by setting $E_F(f)\vert_{K_i}=E_i(f\vert_{F_i})$, if $F_i\ne\emptyset$,
$E_F(f)\vert_{K_i}\equiv f(p)$, if $F_i=\emptyset$, and (if $I$ is infinite) $E_F(f)(\infty)=f(p)$. If $\infty$ belongs to $F$, define $E_F$ by setting
$E_F(f)\vert_{K_i}=E_i(f\vert_{F_i})$, if $F_i\ne\emptyset$, $E_F(f)\vert_{K_i}=f(\infty)$, if $F_i=\emptyset$, and $E_F(f)(\infty)=f(\infty)$.
In the case when $\infty\in F$, the continuity of $E_F(f)$ at the point $\infty$ is easily proven using the fact that $E_i$ is a positive operator.
\end{proof}

%\begin{cor}\label{thm:corTaras}
%If $K$ is a hereditarily paracompact scattered compact Hausdorff space then $K$ has the regular extension property.
%\end{cor}
%\begin{proof}
%It is known \cite[Theorem~3,~(3)]{Taras} that the smallest class of spaces closed under the operation of taking Alexadroff compactifications
%of arbitrary topological sums and containing the singleton is the class of hereditarily paracompact scattered compact Hausdorff spaces.
%\end{proof}

\begin{cor}\label{thm:corheight2}
If $K$ is a scattered compact Hausdorff space of height $2$ then $K$ has the regular extension property.
\end{cor}
\begin{proof}
Such a $K$ is a finite topological sum of one-point compactifications of discrete spaces.
\end{proof}

\subsection{Negative results about the extension property}
We now give a few examples of spaces for which the extension property fails (Examples~\ref{exa:laddersystem}, \ref{exa:2akappa},
and \ref{exa:Mrowka}). We also prove results relating the extension property with cardinal invariants of the topological space $K$
(Corollary~\ref{thm:corccc}) and with the Grothendieck property of the space $C(K)$ (Proposition~\ref{thm:propGrothendieck}).

\begin{example}\label{exa:laddersystem}
The {\em ladder system space\/} \cite[pg.~164]{Ark} is an example of a scattered compact Hausdorff space of height $3$ that does not have the extension property.
We recall the relevant definitions.
Denote by $\omega_1$ the first uncountable ordinal and by $L(\omega_1)$ the subset of $\omega_1$ consisting of limit ordinals.
Set $S(\omega_1)=\omega_1\setminus L(\omega_1)$. Recall that a {\em ladder system\/} on $\omega_1$ is a family $\mathcal S=(s_\alpha)_{\alpha\in L(\omega_1)}$
where, for each $\alpha\in L(\omega_1)$, we have $s_\alpha=\big\{s^n_\alpha:n\in\omega\big\}$ and $(s^n_\alpha)_{n\in\omega}$ is a strictly
increasing sequence in $S(\omega_1)$ order-converging to $\alpha$. The {\em ladder system topology\/} $\tauS$
on $\omega_1$ is the one for which the elements of $S(\omega_1)$ are isolated and the fundamental neighborhoods of a limit
ordinal $\alpha$ are unions of $\{\alpha\}$ with sets cofinite in $s_\alpha$. The ladder system space
is the one-point compactification $\KS=\omegaS\cup\{\infty\}$, where $\omegaS=(\omega_1,\tauS)$. We claim that the closed subset
$F=L(\omega_1)\cup\{\infty\}$ does not admit an extension operator in $\KS$. Assuming by contradiction that an extension operator $E_F$
exists, then, for each $\alpha\in L(\omega_1)$, let $\xi_\alpha\in C(\KS)$ denote the image under $E_F$ of the characteristic function of $\{\alpha\}$.
By the continuity of $\xi_\alpha$, there exists $\phi(\alpha)\in s_\alpha$ with $\xi_\alpha\big(\phi(\alpha)\big)>\frac12$. The map $\phi:L(\omega_1)\to\omega_1$
is regressive and the Pressing Down Lemma (\cite[Lemma~II.6.15]{Kunen}) yields a stationary subset $S$ of $L(\omega_1)$ on which $\phi$ is constant
and equal to some $\beta\in\omega_1$. Then, for each positive integer $n$, we pick distinct elements $\alpha_1,\ldots,\alpha_n\in S$ and
we obtain:
\[\Vert E_F\Vert\ge\Vert\xi_{\alpha_1}+\cdots+\xi_{\alpha_n}\Vert\ge\xi_{\alpha_1}(\beta)+\cdots+\xi_{\alpha_n}(\beta)>\frac n2.\]
\end{example}

\medskip

Recall that a topological space $X$ is said to satisfy the {\em countable chain condition\/} (briefly: $X$ is {\em ccc}) if every
family of disjoint nonempty open subsets of $X$ is countable.
\begin{prop}
Let $K$ be a compact Hausdorff space and $F$ be a closed subset of $K$. If $K$ is ccc and $F$ admits an extension operator in $K$ then $F$ is ccc.
\end{prop}
\begin{proof}
Notice that the image of an extension operator $E_F:C(F)\to C(K)$ is an isomorphic copy of $C(F)$ in $C(K)$. The conclusion follows using the fact
that a compact Hausdorff space $K$ fails to be ccc if and only if $C(K)$ contains an isomorphic copy of $c_0(I)$ for some uncountable
set $I$ (see \cite[Theorem~14.26]{Fabian}).
\end{proof}

\begin{cor}\label{thm:corccc}
If a compact Hausdorff space $K$ has the extension property and is ccc then $K$ has countable spread, i.e., every discrete subset of $K$ is countable.
\end{cor}
\begin{proof}
Notice that if $D$ is an uncountable discrete subset of $K$ then $F=\overline D$ is not ccc, since the points of $D$ are isolated in $F$.
\end{proof}

\begin{example}\label{exa:2akappa}
For an uncountable cardinal $\kappa$, the space $2^\kappa$ does not have the extension property. Namely, $2^\kappa$ is ccc (\cite[Theorem~II.1.9]{Kunen}) and
the characteristic functions of the singleton subsets of $\kappa$ form a discrete subset of $2^\kappa$ of size $\kappa$. In fact, we can show
that $2^\kappa$ is not a continuous image of a compact Hausdorff space having the extension property. Namely, since the ladder system space $\KS$
is a Boolean space of weight $\omega_1$, it is homeomorphic to a subspace of $2^\kappa$. But $\KS$ is not a continuous image of a compact Hausdorff
space having the extension property (Corollary~\ref{thm:continuousimage} and Example~\ref{exa:laddersystem}).
\end{example}

\begin{example}\label{exa:Mrowka}
Another example of a scattered compact Hausdorff space of height $3$ that does not have the extension property is {\em Mr\'owka's space\/}
(the one-point compactification of the space $X$ constructed in \cite{Mrowka}). Namely, Mr\'owka's space is separable (hence ccc) and has
uncountable spread. Then Corollary~\ref{thm:corccc} applies.
\end{example}

Recall that a Banach space $X$ has the {\em Grothendieck property\/} if every bounded operator from $X$ to a separable Banach space
is weakly compact.
\begin{prop}\label{thm:propGrothendieck}
Let $K$ be an infinite compact Hausdorff space. If $C(K)$ has the Grothendieck property then $K$ does not have the extension property. Moreover,
$K$ is not a continuous image of a compact Hausdorff space having the extension property.
\end{prop}
\begin{proof}
If $K$ is infinite then $C(K)$ contains an isometric copy of $c_0$. If $K$ were a continuous image of a compact Hausdorff space having the extension
property then this copy of $c_0$ would be complemented in $C(K)$ (Remark~\ref{thm:remcontinuousimage}).
But every separable complemented subspace of a space with the Grothendieck property is reflexive.
\end{proof}

\end{section}

\end{document}